\documentclass{amsart}
\usepackage[toc,page]{appendix}
\usepackage{geometry}
\usepackage[all,cmtip]{xy}
\usepackage{amsmath}
\usepackage{amsfonts}
\usepackage{amssymb}
\usepackage{mathrsfs}   
\usepackage{amsthm}
\usepackage{xcolor}
\usepackage{cite}
\usepackage{stmaryrd}
\usepackage{graphicx}
\usepackage{pgf}
\usepackage{eepic}
\usepackage{pstricks}
\usepackage[all,cmtip]{xy}
\usepackage{epsfig}
\usepackage{fancyhdr}
\usepackage[backref=page]{hyperref}
\renewcommand*{\backref}[1]{}
\renewcommand*{\backrefalt}[4]{({%
    \ifcase #1 Not cited.%
          \or page~#2%
          \else pages #2%
    \fi%
    })}
    
 \usepackage{mathptmx}  

 \newcommand\imCMsym[4][\mathord]{%
  \DeclareFontFamily{U} {#2}{}
  \DeclareFontShape{U}{#2}{m}{n}{
    <-6> #25
    <6-7> #26
    <7-8> #27
    <8-9> #28
    <9-10> #29
    <10-12> #210
    <12-> #212}{}
  \DeclareSymbolFont{CM#2} {U} {#2}{m}{n}
  \DeclareMathSymbol{#4}{#1}{CM#2}{#3}
}
\newcommand\alsoimCMsym[4][\mathord]{\DeclareMathSymbol{#4}{#1}{CM#2}{#3}}

\imCMsym{cmmi}{124}{\CMjmath}
\imCMsym[\mathop]{cmsy}{113}{\CMamalg}
\imCMsym[\mathop]{cmex}{96}{\CMcoprod}
\alsoimCMsym[\mathop]{cmex}{97}{\CMbigcoprod}
 
\theoremstyle{plain}

\newtheorem{theorem}{Theorem}[section]

\newtheorem{proposition}[theorem]{Proposition}
\newtheorem{corollary}[theorem]{Corollary}

\newtheorem{lemma}[theorem]{Lemma}

\theoremstyle{definition}

\newtheorem{definition}[theorem]{Definition}

\theoremstyle{remark}
\newtheorem{remark}[theorem]{Remark}
\newtheorem{example}[theorem]{Example}

\newtheorem*{warningu}{Warning}

\newcommand{\N}{{\mathbb N}}
\newcommand{\Z}{{\mathbb Z}}
\newcommand{\Q}{{\mathbb Q}}
\newcommand{\R}{{\mathbb R}}
\newcommand{\C}{{\mathbb C}}

\newcommand{\F}{{\mathbb F}}

\newcommand{\G}{{\mathbb G}}
\renewcommand{\P}{{\mathbb P}}

\newcommand{\lser}[1]{(\!(#1)\!)}
\newcommand{\pow}[1]{\llbracket #1 \rrbracket}

\newcommand{\spec}[1]{\mathrm{Spec}\left(#1\right)}
\newcommand{\cur}[1]{\mathcal{#1}}

\newcommand{\norm}[1]{\left\vert#1\right\vert}
\newcommand{\isomto}{\overset{\sim}{\rightarrow}}

\newcommand{\rig}{\mathrm{rig}}

\newcommand{\ekd}{\cur{E}_K^\dagger}

\newcommand{\et}{\mathrm{\acute{e}t}}

\newcommand{\pn}{(\varphi,\nabla)}
\newcommand{\rk}{\cur{R}_K}

\title[Combinatorial degenerations]{Combinatorial degenerations of surfaces and Calabi--Yau threefolds}

\begin{document}

\author{Bruno Chiarellotto}
       \address{Universit\`a degli Studi di Padova \\ Dipartimento di Matematica ``Tullio Levi-Civita'' \\
        Via Trieste, 63 \\ 
        35121 Padova \\ 
        Italia}
       \email{chiarbru@math.unipd.it}

\author{Christopher Lazda}
       \address{Universit\`a degli Studi di Padova \\ Dipartimento di Matematica ``Tullio Levi-Civita'' \\
        Via Trieste, 63 \\ 
        35121 Padova \\ 
        Italia}
       \email{lazda@math.unipd.it}

\fancyhead[RO]{B. Chiarellotto, C. Lazda}
\fancyhead[LE]{Combinatorial degenerations}
\fancyfoot[C]{\thepage}
       
\maketitle 

\begin{abstract} In this article we study combinatorial degenerations of minimal surfaces of Kodaira dimension 0 over local fields, and in particular show that the `type' of the degeneration can be read off from the monodromy operator acting on a suitable cohomology group. This can be viewed as an arithmetic analogue of results of Persson and Kulikov on degenerations of complex surfaces, and extends various particular cases studied by Matsumoto, Liedtke and Matsumoto and Hern\'andez Mada. We also study `maximally unipotent' degenerations of Calabi--Yau threefolds, following Koll\'ar and Xu, showing in this case that the dual intersection graph is a 3-sphere.
\end{abstract}

\tableofcontents

\section{Introduction}

Fix a complete discrete valuation ring $R$ with perfect residue field $k$ of characteristic $p>3$ and fraction field $F$. Let $\pi$ be a uniformiser for $R$, and let $X$ be a smooth and projective scheme over $F$. Let $\overline{F}$ be a separable closure of $F$.

\begin{definition} \label{defmod} A \emph{model} of $X$ over $R$ is a regular algebraic space $\cur{X}$, proper and flat over $\cur{X}$ over $R$, whose generic fibre is isomorphic to $X$, and whose special fibre is a scheme. We say that a model is semistable if it is \'{e}tale locally smooth over $R[x_1,\ldots,x_d](x_1\ldots x_r-\pi)$, and strictly semistable if furthermore the irreducible components of the special fibre $Y$ are smooth over $k$.
\end{definition}

A major question in arithmetic geometry is that of determining criteria under which $X$ has good or semistable reduction over $F$, i.e. admits a model $\cur{X}$ which is smooth and proper over $R$, or semistable over $R$. In general the question of determining good reduction criteria comes in two flavours.
\begin{enumerate} \item Does there exists a model $\cur{X}$ of $X$ which is smooth over $R$?
\item Given a semistable model $\cur{X}$ of $X$, can we tell whether or not $\cur{X}$ is smooth?
\end{enumerate}

We will refer to the first of these as the problem of `abstract' good reduction, and the second as the problem of `concrete' good reduction. The sorts of criteria we expect are those that can be expressed in certain homological or homotopical invariants of the variety in question. In this article we will mainly concentrate on these problems for minimal smooth projective surfaces over $F$ of Kodaira dimension $0$. These naturally fall into four classes:
\begin{itemize} \item K3 surfaces;
\item Enriques surfaces;
\item abelian surfaces;
\item bielliptic surfaces,
\end{itemize}
and in each case we have both the abstract and concrete good reduction problem. Note that for this article we will generally use `abelian surface' to mean a surface over $F$ that is geometrically an abelian surface, i.e. we do not necessarily assume the existence of an $F$-rational point (or thus of a group law).

In the analogous complex analytic situation (i.e. that of a semistable, projective degeneration $X \rightarrow \Delta$ over the open unit disc with general fibre $X_t$ a minimal complex algebraic surface with $\kappa=0$) it was shown by Persson \cite{Per77} and Kulikov \cite{Kul77} that, under a certain (reasonably strong) hypothesis on the total space $X$ one could quite explicitly describe the `shape' of the special fibre, and that these shapes naturally fall into three `types' depending on the nilpotency index of the logarithm of the monodromy on a suitable cohomology group. Our main result here is an analogue of this result in an `arithmetic' context, namely classifying the special fibre of a strictly semistable scheme over $R$ whose generic fibre is a surface of one of the above types, in terms of the monodromy operator on a suitable cohomology group. The exact form of the theorem is somewhat tricky to state simply, so here we content ourselves with providing a rough outline and refer to the body of the article for more detailed statements. 

\begin{theorem}[Theorems \ref{k3main}, \ref{cgre}, \ref{abmain} and \ref{bimain}] \label{vague}Let $X/F$ be a minimal surface with $\kappa=0$, and let $\ell$ be a prime (possibly equal to $p$). Let $\mathscr{X}/R$ be a `minimal' model of $X$ in the sense of Definition \ref{minmod}. Then the special fibre $Y$ of $\mathscr{X}$ is `combinatorial', and moreover there exists an `$\ell$-adic local system' $V_\ell$ on $X$ such that $Y$ is of Type I, II or III as the nilpotency index of a certain monodromy operator on $H^i(X,V_\ell)$ is $1$, $2$ or $3$ respectively.
\end{theorem}

\begin{remark} \begin{enumerate} \item We will not give the definition of `combinatorial' surfaces here, see Definitions \ref{crdk}, \ref{crde}, \ref{crda} and \ref{crdb}.
\item When $\mathrm{char}(F)=0$ or $\mathrm{char}(F)=p\neq \ell$ then the local system $V_\ell$ is a $\Q_\ell$-\'etale sheaf on $X$, and the corresponding cohomology group is $H^i_{\et}(X_{\overline{F}},V_{\ell})$. This is an $\ell$-adic representation of $G_F$, de Rham when $\ell=p$ and $\mathrm{char}(F)=0$, and hence has a monodromy operator attached to it. 
\item When $\mathrm{char}(F)=p=\ell$ then the local system $V_\ell=V_p$ is an overconvergent $F$-isocrystal, and the corresponding cohomology group is a certain form of rigid cohomology $H^i_\rig(X/\rk,V_p)$. This is a $\pn$-module over the Robba ring $\rk$ and hence has a monodromy operator by the $p$-adic local monodromy theorem. For more details on $p$-adic cohomology in equicharacteristic $p$ case see \S\ref{rpad}. \end{enumerate}
\end{remark}

Certain types of results of this sort have been studied before, for example by Matsumoto in \cite{Mat15} (for $\mathrm{char}(F)\neq \ell$ and $X$ a K3 surface), Liedtke and Matsumoto in \cite{LM14} ($\mathrm{char}(F)=0$, $\ell\neq p$ and $X$ K3 or Enriques), Hern\'andez-Mada in \cite{HM15} ($\mathrm{char}(F)=0$, $\ell=p$ and $X$ K3 or Enriques), and P\'erez Buend\'ia in \cite{PB14} ($\mathrm{char}(F)=0$, $\ell=p$ and $X$ K3), and our purpose here is partly to unify these existing results into a broader picture, and partly to fill in various gaps, for example allowing $\ell=p=\mathrm{char}(F)$ in the case of K3 surfaces. It is perhaps worth noting that even treating the case of abelian surfaces is not quite as irrelevant as it may seem (given the rather well-known results on good reduction criteria for abelian varieties) since our result describes the possible shape of the special fibre of a \emph{proper}, but not necessarily smooth model. We also relate these shapes to the more classical description of the special fibre of the N\'eron model, at least after a finite base change (Proposition \ref{typerank}).

In each case (K3, Enriques, abelian, bielliptic) the proof of the theorem is in two parts. The first consists of showing that the special fibre $Y$ is combinatorial, this uses coherent cohomology and some basic (logarithmic) algebraic geometry. The second then divides the possible shapes into types depending on the nilpotency index of a certain monodromy operator $N$, this uses the weight spectral sequence and the weight monodromy conjecture (which in all cases is known for dimensions $\leq2$). Although we do not use it explicitly, constantly lurking in the background here is a Clemens--Schmid type exact sequence of the sort considered in \cite{CT14}. Unfortunately, while the structure of the argument in all 4 cases is similar, we were not able to provide a single argument to cover all of them, hence parts of this article may seem somewhat repetitive.

The major hypothesis in the theorem is `minimality' of the model $\mathscr{X}$, which is more or less the assumption that the canonical divisor $K_{\mathscr{X}}$ of $\mathscr{X}$ is numerically trivial. For K3 surfaces one expects that such models exist (at least after a finite base change), and Matsumoto in \cite{Kul77} showed that this is true if the semistable reduction conjecture is true for K3 surfaces. For abelian surfaces, this argument adapts to show that one does always have such a model after a finite base change (Theorem \ref{concab}), however, for Enriques surfaces there are counterexamples to the existence of such models (see \cite{LM14}) and it seems likely that the same true for bielliptic surfaces. Unfortunately, the methods used by Persson, Kulikov et al. to describe the special fibre when one does not necessarily have these `minimal models' do not seem to be at all adaptable to the arithmetic situation.

Finally, we turn towards addressing similar questions in higher dimensions by looking at certain `maximally unipotent' degenerations of Calabi--Yau threefolds. The inspiration here is the recent work of Koll\'ar and Xu in \cite{KX15} on log Calabi--Yau pairs, using recently proved results on the Minimal Model Program for threefolds in positive characteristic (in particular the existence of Mori fibre spaces from \cite{BW14}). The main result we obtain (Theorem \ref{CY3}) is only part of the story, unfortunately, proceeding any further (at least using the methods of this article) will require knowing that the weight monodromy conjecture holds in the given situation, so is only likely to be currently possible in equicharacteristic. A key part of the proof uses a certain description of the homotopy type (in particular the fundamental group) of Berkovich spaces, which forces us to restrict to models $\mathscr{X}/R$ which are schemes, rather than algebraic spaces. As the example of K3 surfaces shows, however, any result concerning the `abstract' good reduction problem is likely to involve algebraic spaces, and will therefore require methods to handle this case.

\subsection*{Notation and conventions}

Throughout $k$ will be a perfect field of characteristic $p>3$, $R$ will be a complete DVR with residue field $k$ and fraction field $F$, which may be of characteristic $0$ or $p$. We will choose a uniformiser $\pi$ for $F$, and let $\overline{F}$ denote a separable closure. We will denote by $q$ some fixed power of $p$ such that $\F_q\subset k$.

A variety over a field will be a separated scheme of finite type, and when $X$ is proper and $\mathscr{F}$ is a coherent sheaf on $X$ we will write
\[ h^i(X,\mathscr{F})=\dim H^i(X,\mathscr{F})\;\;\;\;\text{and}\;\;\;\;\chi(X,\mathscr{F})=\sum_i (-1)^i h^i(X,\mathscr{F}).\]
We will also write $\chi(X)=\chi(X,\mathcal{O}_X)$, since we always mean coherent Euler-Poincar\'e characteristics (rather than topological ones) this should not cause confusion.

Unless otherwise mentioned, a surface over any field will always mean a smooth, projective and geometrically connected surface. A ruled surface of genus $g$ is a surface $X$ together with a morphism $f:X\rightarrow C$ to a smooth projective surface $C$ of genus $g$, whose generic fibre is isomorphic to $\P^1$. If we let $F$ denote a smooth fibre of $f$ then an $n$-\emph{ruling} of $f$ (for some $n\geq 1$) will be a smooth curve $D \subset X$ such $D\cdot F=n$, a $1$-ruling will be referred to simply as a ruling. 

\section{Review of $p$-adic cohomology in equicharacteristic}\label{rpad}

In this section we will briefly review some of the material from \cite{LP16} on $p$-adic cohomology when $\mathrm{char}(F)=p$, and explain some of the facts alluded to in the introduction, in particular the existence of monodromy operators. We will therefore let $W=W(k)$ denote the ring of Witt vectors of $k$, $K$ its fraction field, and $\sigma$ the $q$-power Frobenius on $W$ and $K$. In this situation, we have an isomorphism $F\cong k\lser{\pi}$ where $\pi$ is our choice of uniformiser. We will let $\rk$ denote the Robba ring over $K$, that is the ring of series $\sum_i a_it^i$ with $a_i\in K$ such that:
\begin{itemize} \item for all $\rho<1$, $\norm{a_i}\rho^i\rightarrow 0$ as $i\rightarrow \infty$;
\item for some $\eta<1$, $\norm{a_i}\eta^i\rightarrow 0$ as $i\rightarrow -\infty$.
\end{itemize}
In other words, it is the ring of functions convergent on some semi-open annulus $\eta\leq \norm{t}<1$. The ring of integral elements $\rk^{\mathrm{int}}$ (i.e. those with $a_i\in W$) is therefore a lift of $F$ to characteristic $0$, in the sense that mapping $t\mapsto \pi$ induces $\rk^{\mathrm{int}}/(p)\cong F$. We will denote by $\sigma$ a Frobenius on $\rk$, i.e. a continuous $\sigma$-linear endomorphism preserving $\rk^{\mathrm{int}}$ and lifting the absolute $q$-power Frobenius on $F$, we will moreover assume that $\sigma(t)=ut^q$ for some $u\in (W\pow{t}\otimes_W K)^\times$. The reader is welcome to assume that $\sigma(\sum_i a_it^i)=\sum_i \sigma(a_i)t^{iq}$. Let $\partial_t:\rk\rightarrow \rk$ denote the derivation given by differentiation with respect to $t$.

\begin{definition} \label{pnm1} A $\pn$-module over $\rk$ is a finite free $\rk$-module $M$ together with:
\begin{itemize} \item a connection, that is a $K$-linear map $\nabla:M\rightarrow M$ such that
\[ \nabla(rm)=\partial_t(r)m+r\nabla(m)\;\;\;\;\text{for all}\;\; r\in \rk \;\; \text{and} \;\;m\in M;\]
\item a horizontal Frobenius $\varphi:\sigma^*M:=M\otimes_{\rk,\sigma} \rk \isomto M$.
\end{itemize}
\end{definition}

Then $\pn$-modules over $\rk$ should be considered as $p$-adic analogues of Galois representations, for example, they satisfy a local monodromy theorem (see \cite{Ked04a}) and hence have a canonical monodromy operator $N$ attached to them (see \cite{Mar08}). More specifically, the connection $\nabla$ should be viewed as an analogue of the action of the inertia subgroup $I_F$ and the Frobenius $\varphi$ the action of some Frobenius lift in $G_F$. The analogue for $\pn$-modules of inertia acting unipotently (on an $\ell$-adic representation for $\ell\neq p$) or of a $p$-adic Galois representation being semistable (when $\mathrm{char}(F)=0$) is therefore the connection acting unipotently, i.e. there being a basis $m_1,\ldots,m_n$ such that $\nabla(m_i)\in \rk m_1+\ldots+\rk m_{i-1}$ for all $i$. The analogue of being unramified or crystalline for a $\pn$-module $M$ is therefore the connection acting trivially, or in other words $M$ admitting a basis of horizontal sections. We call such $\pn$-modules $M$ solvable.

Let $\mathcal{E}_K^\dagger\subset \rk$ denote the bounded Robba ring, that is the subring consisting of series $\sum_i a_it^i$ such that $\norm{a_i}$ is bounded, we therefore have the notion of a $\pn$-module over $\ekd$, as in Definition \ref{pnm1}.  The main purpose of the book \cite{LP16} was to define cohomology groups
\[  X\mapsto H^i_\rig(X/\ekd) \]
for $i\geq0$ associated to any $k\lser{\pi}$-variety $X$ (i.e. separated $k\lser{\pi}$-scheme of finite type), as well as versions with compact support $H^i_{c,\rig}(X/\ekd)$ or support in a closed subscheme $Z\subset X$, $H^i_{Z,\rig}(X/\ekd)$. These are $\pn$-modules over $\ekd$ and enjoy all the same formal properties as $\ell$-adic \'etale cohomology for $\ell\neq p$. Here we list a few of them.
\begin{enumerate} \item If $X$ is of dimension $d$ then $H^i_\rig(X/\ekd)=H^i_{c,\rig}(X/\ekd)=H^i_{Z,\rig}(X/\ekd)=0$ for $i$ outside the range $0\leq i\leq 2d$.
\item (K\"unneth formula) For any $X,Y$ over $k\lser{\pi}$ we have
\[ H^n_{c,\rig}(X\times Y/\ekd) \cong \bigoplus_{i+j=n} H^i_{c,\rig}(X/\ekd)\otimes_{\ekd} H^j_{c,\rig}(Y/\ekd) \]
and if $X$ and $Y$ are smooth over $k\lser{\pi}$ we also have
\[ H^n_{\rig}(X\times Y/\ekd) \cong \bigoplus_{i+j=n} H^i_{\rig}(X/\ekd)\otimes_{\ekd} H^j_{\rig}(Y/\ekd).\]
\item (Poincar\'e duality) For any $X$ smooth over $k\lser{\pi}$ of equidimension $d$ we have a perfect pairing
\[ H^i_\rig(X/\ekd) \times H^{2d-i}_{c,\rig}(X/\ekd) \rightarrow H^{2d}_{c,\rig}(X/\ekd)\cong \ekd(-d) \]
where $(-d)$ is the Tate twist which multiplies the Frobenius structure on the constant $\pn$-module $\ekd$ by $q^d$.
\item (Excision) For any closed $Z\subset X$ with complement $U\subset X$ we have long exact sequences
\[ \ldots \rightarrow H^i_{Z,\rig}(X/\ekd) \rightarrow H^i_\rig(X/\ekd) \rightarrow H^i_\rig(U/\ekd) \rightarrow \ldots \]
and 
\[ \ldots \rightarrow H^i_{c,\rig}(U/\ekd) \rightarrow H^i_{c,\rig}(X/\ekd) \rightarrow H^i_{c,\rig}(Z/\ekd) \rightarrow \ldots.\]
\item (Gysin) For any closed immersion $Z\hookrightarrow X$ of smooth schemes over $k\lser{\pi}$, of constant codimension $c$ there is a Gysin isomorphism \[ H^i_{Z,\rig}(X/\ekd)\cong H^{i-2c}_\rig(Z/\ekd)(-c). \]
\item There is a `forget supports' map $H^i_{c,\rig}(X/\ekd)\rightarrow H^i_\rig(X/\ekd)$ which is an isomorphism whenever $X$ is proper over $k\lser{\pi}$.
\item Let $U\subset C$ be an open subcurve of a smooth projective curve $C$ of genus $g$, with complementary divisor $D$ of degree $d$. Then 
\[ \dim_{\ekd} H^1_\rig(U/\ekd)= \begin{cases} 2g-1+d & \text{if } d\geq 1, \\
2g & \text{if }d=0.
\end{cases} \]
\item Let $A$ be an abelian variety over $k\lser{\pi}$ of dimension $g$. Then $H^1_\rig(A/\ekd)$ is (more or less) isomorphic to the contravariant Dieudonn\'e module of the $p$-divisible group $A[p^\infty]$ of $A$, has dimension $2g$, and
\[H^i_\rig(A/\ekd)\cong \textstyle{\bigwedge^i} H^1_\rig(A/\ekd).\]
\end{enumerate}
All of these properties were proved in \cite{LP16}. We may therefore define, for any variety $X/k\lser{\pi}$
\[ H^i_\rig(X/\rk) := H^i_\rig(X/\ekd)\otimes_{\ekd}\rk \]
as $\pn$-modules over $\rk$. That the property of a $\pn$-module being solvable (resp. unipotent) really is the correct analogue of a Galois representation being unramified or crystalline (resp. unipotent or semistable) is suggested by the following result. 

\begin{theorem}[\cite{LP16}, \S5] Let $X/k\lser{\pi}$ be smooth and proper. Then if $X$ has good (resp. semistable reduction) then $H^i_\rig(X/\rk)$ is solvable (resp. unipotent) for all $i\geq0$. If moreover $X$ is an abelian variety, then the converse also holds.
\end{theorem}

In \cite{LP16} was also shown an equicharacteristic analogue of the $C_{\mathrm{st}}$-conjecture, namely that when $\mathcal{X}/R$ is proper and semistable, the cohomology $H^i_\rig(X/\rk)$ of the generic fibre can be recovered from the log-crystalline cohomology $H^i_{\log\text{-}\mathrm{cris}}(Y^{\log}/W^{\log})\otimes_W K$ of the special fibre. Our task for the remainder of this section is to generalise this result to algebraic spaces (with fibres that are schemes).

So fix a smooth and proper variety $X/F$ and a semistable model $\mathcal{X}/R$ (see Definition \ref{defmod}) for $X$. Let $Y^{\log}$ denote the special fibre of $\mathcal{X}$ with its induced log structure, and let $W^{\log}$ denote $W$ with the log structure defined by $1\mapsto 0$. Then the log-crystalline cohomology $H^i_{\log\text{-}\mathrm{cris}}(Y^{\log}/W^{\log})\otimes_W K$ is a $(\varphi,N)$-module over $K$, i.e. a vector space with semilinear Frobenius $\varphi$ and nilpotent monodromy operator $N$ satisfying $N\varphi=q\varphi N$, and the rigid cohomology $H^i_\rig(X/\rk)$ is a $\pn$-module over $\rk$. There is a fully faithful functor
\[ (-)\otimes_K \rk: \underline{\mathbf{M}\Phi}^N_K\rightarrow  \underline{\mathbf{M}\Phi}^\nabla_{\rk} \]
from the category $\underline{\mathbf{M}\Phi}^N_K$ of $(\varphi,N)$-modules over $K$ to that of $\pn$-modules over $\rk$, whose essential image consists exactly of the unipotent $\pn$-modules, i.e. those which are iterated extensions of constant ones. The analogue of Fontaine's $C_{\mathrm{st}}$ conjecture in the equicharacteristic world is then the following.

\begin{proposition}\label{cstpas} There is an isomorphism \[\left(H^i_{\log\text{-}\mathrm{cris}}(Y^{\log}/W^{\log})\otimes_W K\right)\otimes_K \rk \cong H^i_\rig(X/\rk)\]
in $ \underline{\mathbf{M}\Phi}^\nabla_{\rk} $.
\end{proposition}

\begin{proof} Thanks to the extension of logarithmic crystalline cohomology and Hyodo--Kato cohomology to algebraic stacks by Olsson in \cite{Ols07a}, in particular base change (Theorem 2.6.2) and the construction of the monodromy operator (\S6.5), the same proof as given in the scheme case (see Chapter 5 of \cite{LP16}) works for algebraic spaces as well.
\end{proof}

In \cite{LP16} was defined the notion of an overconvergent $F$-isocrystal on $X$, relative to $K$. These play the role in the $p$-adic theory of lisse $\ell$-adic sheaves in $\ell$-adic cohomology. Classically, i.e. over $k$, one can associate these objects to $p$-adic representations of the fundamental group, and we will need to do this also over Laurent series fields. We only need this for representations $\rho$ with finite image, and in this case the construction is simple. So let $\rho:\pi_1^{\et}(X,\overline{x})\rightarrow G$ be a finite quotient of the \'etale fundamental group of a smooth and proper variety over $F$, then this corresponds to a finite, \'etale, Galois cover $f:X'\rightarrow X$, and hence from results of \cite{LP16} we have a pushforward functor
\[ f_*:F\text{-}\mathrm{Isoc}^\dagger(X'/K)\rightarrow F\text{-}\mathrm{Isoc}^\dagger(X/K)\]
from overconvergent $F$-isocrystals on $X'$ to those on $X$. We may therefore define $V_\rho\in F\text{-}\mathrm{Isoc}^\dagger(X/K)$ to be the pushforward $f_*\mathcal{O}^\dagger_{X'/K}$ of the constant isocrystal on $X'$.

\section{SNCL varieties}\label{snclv}

In this section, following F. Kato in \S11 of \cite{Kat96}, we will introduce the key notion of a simple normal crossings log variety over $k$, or SNCL variety for short. 

\begin{definition} We say a geometrically connected variety $Y/k$ is a normal crossings variety over $k$ if it is \'etale locally \'etale over $k[x_0,\ldots,x_d]/(x_0\cdots x_r)$. 
\end{definition}

\begin{definition} Let $Y$ denote a normal crossings variety over $k$, and let $M_Y$ be a log structure on $Y$. Then we say that $M_Y$ is of embedding type if \'etale locally on $Y$ it is (isomorphic to) the log structure associated to the homomorphism of monoids
\[\N^{r+1} \rightarrow  \frac{k[x_0,\ldots,x_d]}{(x_0\cdots x_r)}\]
sending the $i$th basis element of $\N^{r+1}$ to $x_i$.
\end{definition}

Note that the existence of such a log structure imposes conditions on $Y$, and the log structure $M_Y$ is \emph{not} determined by the geometry of the underlying scheme $Y$. In fact, one can show that such a log structure exists if and only if, denoting by $D$ the singular locus of $Y$, there exists a line bundle $\mathcal{L}$ on $Y$ such that $\mathcal{E}\mathit{xt}^1(\Omega^1_{Y/k},\mathcal{O}_Y) \cong \mathcal{L}\otimes \mathcal{O}_D$ (see for example Theorem 11.7 of \cite{Kat96}).

\begin{definition} We say that a log scheme $Y^{\log}$ of embedding type is of semistable type if there exists a log smooth morphism $Y^{\log}\rightarrow \spec{k}^{\log}$ where the latter is endowed with the log structure of the punctured point. 
\end{definition}

Again, the existence of such a morphism implies conditions on $Y$, namely that $\mathcal{E}\mathit{xt}^1(\Omega^1_{Y/k},\mathcal{O}_Y) \cong  \mathcal{O}_D$ (where again $D$ is the singular locus).

\begin{definition} A SNCL variety over $k$ is a smooth log scheme $Y^{\log}$ over $k^{\log}$ of semistable type, such that the irreducible components of $Y$ are all smooth.
\end{definition}

Any SNCL variety $Y^{\log}$ is log smooth over $k^{\log}$ by definition, and for all $p\geq0$ we will let $\Lambda^p_{Y^{\log}/k^{\log}}$ denote the locally free sheaf of logarithmic $p$-forms on $Y$. We will also let $\omega_Y=\Lambda^{\dim Y}_{Y^{\log}/k^{\log}}$ denote the line bundle of top degree differential forms.

\begin{proposition} \label{dualsheaf} The sheaf $\omega_Y$ is a dualising sheaf for $Y$.
\end{proposition}

\begin{proof} Follows immediately from Proposition 2.14 and Theorem 2.21 of \cite{Tsu99a}.
\end{proof}

We will also need a spectral sequence for the cohomology of semistable varieties. This should be well-known, but we could not find a suitable reference.

\begin{lemma} \label{sslc} Let $Y^{\log}$ be a SNCL variety over $k$ of dimension $n$, with smooth components $Y_1,\ldots,Y_N$. For each $0\leq s\leq n$ write 
\[ Y^{(s)}=\CMcoprod_{\substack{ I\subset \{1,\ldots,N\}\\ \norm{I}=s+1}} \bigcap_{i\in I} Y_i,\]
and let $i_s:Y^{(s)}\rightarrow Y$ denote the natural map. For $1\leq t\leq s+1$ let 
\[ \partial^s_t: Y^{(s+1)}\rightarrow Y^{(s)} \]
be the canonical map induced by the natural inclusion $Y_{\{i_1,\ldots,i_{s+1}\}}\rightarrow Y_{\{i_1,\ldots,\hat{i}_t,\ldots,i_{s+1}\}}$. Then the there exists an exact sequence
\[  0 \rightarrow \cur{O}_Y\overset{d^{-1}}{\rightarrow} i_{0*} \cur{O}_{Y^{(0)}}\overset{d^0}{\rightarrow} \ldots \overset{d^{n-1}}{\rightarrow} i_{n*}\cur{O}_Y^{(n)} \rightarrow 0 \]
of sheaves on $Y$ where $d^{-1}=i_0^*$ and 
\[ d^s = \sum_{t=1}^{s+1} (-1)^t \partial_t^{s*}\]
for $s\geq 0$.
\end{lemma}

\begin{proof} We define a complex
\[ 0 \rightarrow \cur{O}_Y\rightarrow i_{0*} \cur{O}_{Y^{(0)}}\rightarrow \ldots \rightarrow i_{n*}\cur{O}_Y^{(n)} \rightarrow 0 \]
using the formulae in the statement of the lemma, to check it is in fact exact (or indeed, to check that it is even a complex) we may work locally, and hence assume that $Y$ is smooth over $\spec{\frac{k[x_1,\ldots,x_{d}]}{(x_1\ldots x_r)}}$. But now we can just use flat base change to reduce to the case where $Y=\spec{\frac{k[x_1,\ldots,x_{d}]}{(x_1\ldots x_r)}}$, which follows from a straightforward computation.\end{proof}

\begin{corollary} In the above situation, there exists a spectral sequence
\[ E^{s,t}_1:=H^t(Y^{(s)},\cur{O}_{Y^{(s)}})\Rightarrow H^{s+t}(Y,\cur{O}_Y). \]
\end{corollary}

\section{Some useful results}

In this section we prove three lemmas that will come in handy later on. The first characterises surfaces with effective anticanonical divisor of a certain form, analogous to Lemma 3.3.7 of \cite{Per77} in the complex case.

\begin{lemma} \label{paac} Let $k$ be an algebraically closed field, and $V$ a surface with canonical divisor $K_V$. Let $\{C_i\}$ be a non-empty family of smooth curves $C_i$ on $V$, such that the divisor $D=\sum_iC_i$ is a simple normal crossings divisor, and we have $K_V+D=0$ in $\mathrm{Pic}(V)$. Then one of the following must happen.
\begin{enumerate}
\item $V$ is an elliptic ruled surface, and $D=E_1+E_2$ is a sum of disjoint elliptic curves, which are rulings on $V$.
\item $V$ is an elliptic ruled surface, and $D=E$ is a single elliptic curve, which is a 2-ruling on $V$,
\item $V$ is rational, and $D=E$ is an elliptic curve.
\item $V$ is rational, and $D=\sum_{i=1}^d C_i$ is a cycle of rational curves on $V$, i.e. either $d=2$ and $C_1\cdot C_2=2$, or $d>2$ and $C_1\cdot C_2 = C_2\cdot C_3 =\ldots = C_d \cdot  C_1=1$, with all other intersection numbers $0$.
\end{enumerate}
\end{lemma}

\begin{proof} The point is that since the classification of surfaces is essentially the same in characteristic $p$ as characteristic $0$, Persson's original proof carries over verbatim. We reproduce it here for the reader's benefit.

The hypotheses imply that $V$ is of Kodaira dimension $-\infty$, and hence is either rational or ruled. For each curve $C_i$, let $T_{C_i}$ denote the number of double points on $C_i$, that is $\sum_{j\neq i} C_i\cdot C_j$. By the genus formula we have
\[ 2g(C_i)-2 = C_i\cdot (C_i+K_V ) = -T_{C_i} \]
(here $K_V$ is the canonical divisor) and hence either $T_{C_i}=0$ and $g(C_i)=1$ or $T_{C_i}=2$ and $g(C_i)=0$. Hence $D$ is a disjoint sum of elliptic curves and cycles of rational curves.

Let $\pi:V\rightarrow V_0$ be a map onto a minimal model. For any $i$ such that $\pi$ does not contract $C_i$, let $C_{0i}=\pi(C_i)$, and let $D_0:=\pi(D)$. Any exceptional curve $E$ has to either be a component of a rational cycle or meet exactly one component of $D$ in exactly one point (because $D\cdot E=-K_V\cdot E = 1$). It then follows that $D_0$ has the same form as $D$ (i.e. is a disjoint union of elliptic curves and cycles of rational curves) except that it might also contain nodal rational curves, not meeting any other components. If $V_0\cong \P^2$, then the only possibilities for $D_0$ are a triangle of lines, a conic plus a line, a single elliptic curve or a nodal cubic. Therefore $(V,D)$ has the form claimed. 

Otherwise, $V_0$ is a $\P^1$ bundle over a smooth projective curve, let $F \subset V_0$ be a fibre intersecting all $C_{0i}$ properly. Applying the genus formula again gives $K_{V_0}\cdot F=-2$, hence $D_0\cdot F=2=\sum_i C_{0i} \cdot F$. Each connected component of $D_0$ is either a rational cycle, a nodal rational curve or an elliptic curve, and the first two kinds of components have to intersect $F$ with multiplicity $\geq 2$ (in the second case this is because it cannot be either a fibre or a degree 1 cover of the base). Hence if some $C_{0i}$ is an elliptic curve $E_1$, then either $E_1\cdot F=2$, in which case $D_0=E_1$, or $E_1\cdot F=1$, in which case we must have $D_0=E_1+E_2$ for some other elliptic curve $E_2$. In the first case $V_0$ can be elliptic ruled, in which case $E_1$ is a 2-ruling, or rational. In the second case $V_0$ must be elliptic ruled, and both $E_1$ and $E_2$ are rulings. Otherwise, each $C_{0i}$ is a rational curve, $V_0$ must be rational and $D_0$ is either a single cycle of smooth rational curves or a single nodal rational curve. Again, this implies that $(V,D)$ has the form claimed.
\end{proof}

We will also need the following cohomological computation.

\begin{lemma} \label{betti} \begin{enumerate} 
\item Let $V$ be an elliptic ruled surface over $k$, and let $\ell$ be a prime number $\neq p$. Then $\dim_{\Q_\ell}H^1_\et(V_{\overline{k}},\Q_\ell)=\dim_K H^1_\rig(V/K)=2$. 
\item Let $V$ be a rational surface over $k$, and let $\ell$ be a prime number $\neq p$. Then $\dim_{\Q_\ell}H^1_\et(V_{\overline{k}},\Q_\ell)=\dim_K H^1_\rig(V/K)=0$.
\end{enumerate} 
\end{lemma}

\begin{proof} One may use the excision exact sequence in either rigid or $\ell$-adic \'etale cohomology to see that the first Betti number of a smooth projective surface is unchanged under monoidal transformations, and is hence a birational invariant. We may therefore reduce to the case of $E\times\P^1$ or $\P^1\times\P^1$, which follows from the K\"unneth formula.
\end{proof}

Finally, we have the following (well known) result.

\begin{lemma} \label{gcggcs} Let $\mathcal{X}/R$ be proper and flat. Assume that the generic fibre $X$ is geometrically connected. Then so is the special fibre $Y$.
\end{lemma}

\begin{proof} Since $\mathcal{X}$ is proper and flat over $R$, the zeroth cohomology $H^0(\mathcal{X},\mathcal{O}_{\mathcal{X}})$ is torsion free and finitely generated over $R$, hence it is free. Since the generic fibre is geometrically connected, it is of rank 1, and the natural map $R\rightarrow H^0(\mathcal{X},\mathcal{O}_{\mathcal{X}})$ is an isomorphism. Since this also holds after any finite flat base change $R\rightarrow R'$, it follows from Zariski's Main Theorem \cite[\href{http://stacks.math.columbia.edu/tag/0A1C}{Tag 0A1C}]{stacks} that $Y$ must in fact be geometrically connected.
\end{proof}

\section{Minimal models, logarithmic surfaces and combinatorial reduction}

The purpose of this section is to introduce the notion of a minimal model of a surface of Kodaira dimension $0$, as well as the corresponding logarithmic and combinatorial versions of these surfaces. The basic idea in all cases is that we have
\[  \text{minimal}\Rightarrow \text{logarithmic}\Rightarrow \text{combinatorial} \]
and although the general form that the picture takes is the same in all 4 cases, there are enough differences to merit describing how it works separately in each case. This unfortunately means that the next few sections are somewhat repetitive.

Let $X/F$ be a smooth, projective, geometrically connected minimal surface of Kodaira dimension 0, and denote the canonical sheaf by $\omega_X$. Then $X$ falls into one of the following four cases.
\begin{enumerate} \item $\omega_X\cong \mathcal{O}_X$ and $h^1(X,\mathcal{O}_X)=0$. Then $X$ is a K3 surface.
\item $h^0(X,\omega_X)=0$ and $h^1(X,\mathcal{O}_X)=0$. Then $X$ is an Enriques surface.
\item $\omega_X\cong \mathcal{O}_X$ and $h^1(X,\mathcal{O}_X)=2$. Then $X$ is an abelian surface.
\item $h^0(X,\omega_X)=0$ and $h^1(X,\mathcal{O}_X)=1$. Then $X$ is a bielliptic surface.
\end{enumerate}

Note that if $X$ is an Enriques surface we have $\omega_X^{\otimes 2}\cong \mathcal{O}_X$ and if $X$ is a bielliptic surface we have $\omega_X^{\otimes m}\cong\mathcal{O}_X$ for $m=2,3,4$ or $6$. Also note that since $p>3$ the classification of such surfaces is the same over $k$ as over $F$ (i.e. we do not have to consider the `extraordinary' Enriques or bielliptic surfaces). In all cases we may therefore define an integer $m$ as the smallest positive integer such that $\omega_X^{\otimes m}\cong \mathcal{O}_X$. If $\mathscr{X}/R$ is a semistable model for $X$ then we will let $\mathscr{X}^{\log}$ denote the log scheme with log structure induced by the special fibre, this is log smooth over $R^{\log}$, where the log structure is again induced by the special fibre $\pi=0$. We will let $\omega_\mathscr{X}=\Lambda^2_{\mathcal{X}^{\log}/R^{\log}}$ denote the line bundle of logarithmic 2-forms on $\mathscr{X}$. We will also let $Y$ denote the special fibre, and $Y^{\log}/k^{\log}$ the smooth log scheme whose log structure is the one pulled back from that on $\mathscr{X}$.

\begin{definition} \label{minmod} Let $\mathscr{X}/R$ be a semistable model for $X$. Then we say that $\mathscr{X}$ is minimal if it is strictly semistable and $\omega_{\mathscr{X}}^{\otimes m}\cong \mathcal{O}_{\mathscr{X}}$.
\end{definition}

\begin{warningu} When $X$ is an Enriques surface, there are counter-examples to the existence of such minimal models, even allowing for finite extensions of $R$.
\end{warningu}

The first stage is in passing from minimal models to logarithmic surfaces of Kodaira dimension $0$, the latter being defined by logarithmic analogues of the above criteria.

\begin{definition} \label{logs} Let $Y^{\log}/k^{\log}$ be a proper SNCL scheme over $k$, of dimension $2$, and let $\omega_Y=\Lambda^2_{Y^{\log}/k^{\log}}$ be its canonical sheaf. Then we say that $Y^{\log}$ is a:
\begin{enumerate} \item logarithmic K3 surface if $\omega_Y\cong \mathcal{O}_Y$ and $h^1(Y,\mathcal{O}_Y)=0$;
\item logarithmic Enriques surface if $\omega_Y$ is torsion in $\mathrm{Pic}(Y)$, $h^0(Y,\omega_Y)=0$ and $h^1(Y,\mathcal{O}_Y)=0$;
\item logarithmic abelian surface if $\omega_Y\cong \mathcal{O}_Y$ and $h^1(Y,\mathcal{O}_Y)=2$;
\item logarithmic bielliptic surface if $\omega_Y$ is torsion in $\mathrm{Pic}(Y)$, $h^0(Y,\omega_Y)=0$ and $h^1(Y,\mathcal{O}_Y)=1$;
\end{enumerate}
\end{definition}

\begin{proposition} Let $X/F$ be a minimal surface of Kodaira dimension $0$, and $\mathscr{X}/R$ a minimal model. Then $Y^{\log}$ is a logarithmic K3 (resp. Enriques, abelain, bielliptic) surface if $X$ is K3 (resp. Enriques, abelian, bielliptic).
\end{proposition}

\begin{proof} Note that the only obstruction to $Y^{\log}/k^{\log}$ being an SNCL variety is geometric connectedness, which follows from Lemma \ref{gcggcs}. The conditions on the canonical sheaf $\omega_Y$ in Definition \ref{logs} follow from the definition of minimality, it therefore suffices to verify the required dimensions of the coherent cohomology groups on $Y$. We divide into the four cases.

First assume that $X$ is a K3 surface. Then we have $\chi(X,\mathcal{O}_X)=2$, and hence by local constancy of $\chi$ under a flat map (see Chapter III, Theorem 9.9 of \cite{Har77}) we must also have that $\chi(Y,\mathcal{O}_Y)=2$. Since $Y$ is geometrically connected by Lemma \ref{gcggcs}, we have $h^0(Y,\mathcal{O}_Y)=1$, and therefore $h^2(Y,\mathcal{O}_Y)-h^1(Y,\mathcal{O}_Y)=1$. But by Proposition \ref{dualsheaf} we must have $h^2(Y,\mathcal{O}_Y)=h^0(Y,\omega_Y)$, and by definition of minimality we know that $\omega_Y\cong \mathcal{O}_Y$. Hence $h^2(Y,\mathcal{O}_Y)=1$ and therefore $h^1(Y,\mathcal{O}_Y)=0$. Hence $Y^{\log}$ is a logarithmic K3 surface.

Next assume that $X$ is Enriques. Then as above, we have that $h^0(Y,\mathcal{O}_Y)=1$ and hence by local constancy of $\chi$, that $h^1(Y,\mathcal{O}_Y)=h^2(Y,\mathcal{O}_Y)$. Let $\pi:\widetilde{\mathscr{X}}\rightarrow \mathscr{X}$ denote the canonical double cover coming from the 2-torsion element $\omega_\mathscr{X}\in \mathrm{Pic}(\mathscr{X})$, with generic fibre $\widetilde{X}\rightarrow X$ and special fibre $\widetilde{Y}\rightarrow Y$. Then $\widetilde{\mathscr{X}}$ is a minimal model of the K3 surface $\widetilde{X}$, and hence $\widetilde{Y}^{\log}$ is a logarithmic K3 surface. Hence $h^1(\widetilde{Y},\mathcal{O}_{\widetilde{Y}})=0$, and since $\mathcal{O}_Y\subset \pi_*\mathcal{O}_{\widetilde{Y}}$ is a direct summand, we must have $h^1(Y,\mathcal{O}_Y) =0$, and therefore $h^0(Y,\omega_Y)=h^2(Y,\mathcal{O}_Y)=0$. Thus $Y^{\log}$ is a logarithmic Enriques surface.

The case of abelian surfaces is handled entirely similarly to that of K3 surfaces, and the case of bielliptic surfaces is then deduced as Enriques is deduced from K3.
\end{proof}

The next notion is that of combinatorial versions of the above four cases.

\begin{definition} \label{crdk} Let $Y$ be a proper surface over $k$ (not necessarily smooth). We say that $Y$ is a combinatorial K3 surface if, geometrically (i.e. over $\overline{k}$), one of the following situations occurs:
\begin{itemize}
\item (Type I) $Y$ is a smooth K3 surface.
\item (Type II) $Y=Y_1\cup \ldots \cup Y_N$ is a chain with $Y_1,Y_N$ smooth rational surfaces and all other $Y_i$ elliptic ruled surfaces, with each double curve on each `inner' component a ruling. The dual graph of $Y_{\overline{k}}$ is a straight line with endpoints $Y_1$ and $Y_N$.
\item (Type III) $Y$ is a union of smooth rational surfaces, the double curves on each component form a cycle of rational curves, and the dual graph of $Y_{\overline{k}}$ is a triangulation of $S^2$.
\end{itemize}
\end{definition}

\begin{definition} \label{crde} Let $Y$ be a proper surface over $k$ (not necessarily smooth). We say that $Y$ is a combinatorial Enriques surface if, geometrically, one of the following situations occurs:
\begin{itemize}
\item (Type I) $Y$ is a smooth Enriques surface.
\item (Type II) $Y=Y_1\cup \ldots \cup Y_N$ is a chain of surfaces, with $Y_1$ rational and all others elliptic ruled, with each double curve on each `inner' component a ruling and the double curve on $Y_N$ a 2-ruling. The dual graph of $Y_{\overline{k}}$ is a straight line with endpoints $Y_1$ and $Y_N$.
\item (Type III) $Y$ is a union of smooth rational surfaces, the double curves on each component form a cycle of rational curves, and the dual graph of $Y_{\overline{k}}$ is a triangulation of $\P^2(\R)$.
\end{itemize}
\end{definition}

\begin{definition} \label{crda} Let $Y$ be a proper surface over $k$ (not necessarily smooth). We say that $Y$ is a combinatorial abelian surface if, geometrically, one of the following situations occurs:
\begin{itemize}
\item (Type I) $Y$ is a smooth abelian surface.
\item (Type II) $Y=Y_1\cup \ldots \cup Y_N$ is a cycle of elliptic ruled surfaces, with each double curve a ruling. The dual graph of $Y_{\overline{k}}$ is a circle.
\item (Type III) $Y$ is a union of smooth rational surfaces, the double curves on each component form a cycle of rational curves, and the dual graph of $Y_{\overline{k}}$ is a triangulation of the torus $S^1\times S^1$.
\end{itemize}
\end{definition}

\begin{definition} \label{crdb} Let $Y$ be a proper surface over $k$ (not necessarily smooth). We say that $Y$ is a combinatorial bielliptic surface if, geometrically, one of the following situations occurs:
\begin{itemize}
\item (Type I) $Y$ is a smooth bielliptic surface.
\item (Type II) $Y=Y_1\cup \ldots \cup Y_N$ is either a cycle or chain of elliptic ruled surfaces, with each double curve either a ruling (cycles or `inner' components of a chain) or a 2-ruling (`end' components of a chain). The dual graph of $Y_{\overline{k}}$ is either a circle or a line segment.
\item (Type III) $Y$ is a union of smooth rational surfaces, the double curves on each component form a cycle of rational curves, and the dual graph of $Y_{\overline{k}}$ is a triangulation of the Klein bottle.
\end{itemize}
\end{definition}

Of course, in each case logarithmic surfaces will turn out to be combinatorial, this has been proved by Nakkajima for K3 and Enriques surfaces, and we will show it during the course of this article for abelian (Theorem \ref{crass}) and bielliptic (Theorem \ref{hst}) surfaces. 

\section{K3 surfaces}

In this section, we will properly state and prove Theorem \ref{vague} for K3 surfaces. The case when $\mathrm{char}(F)=0$ and $\ell=p$ is due to Hern\'andez-Mada in \cite{HM15}, and Perez Buend\'ia in \cite{PB14} and the case $\ell\neq p$ should be well-known (and at least part of it is implicitly proved in \cite{Kul77}), however, we could not find a reference in the literature so we include a proof here for completeness. We begin with a result of Nakkajima. 

\begin{theorem}[\cite{Nak00}, \S3] Let $Y^{\log}$ be a logarithmic K3 surface over $k$. Then the underlying scheme $Y$ is a combinatorial K3 surface.
\end{theorem}

\begin{remark} A proof of this result given entirely in terms of coherent cohomology can be given as in Theorem \ref{crass} below.
\end{remark}

\begin{corollary} Let $\mathscr{X}/R$ be a minimal semistable model of a K3 surface $X/F$. Then the special fibre $Y$ is a combinatorial K3 surface.
\end{corollary}

For a K3 surface $X/K$, and for all $\ell\neq p$, the second cohomology group $H^2_\et(X_{\overline{F}},{\Q_\ell})$ is a finite dimensional $\Q_\ell$ vector space with a continuous Galois action, which is quasi-unipotent. If $\ell=p$ and $\mathrm{char}(F)=0$ then $H^2_\et(X_{\overline{F}},{\Q_p})$ is a de Rham representation of $G_F$, and if $\mathrm{char}(F)=p$ then $H^2_\rig(X/\rk)$ is a $\pn$-module over $\rk$.

If we therefore let $H^2(X)$ stand for:
\begin{itemize} \item $H^2_\et(X_{\overline{F}},{\Q_\ell})$ if $\ell\neq p$;
\item $H^2_\et(X_{\overline{F}},{\Q_p})$ if $\ell=p$ and $\mathrm{char}(F)=0$;
\item $H^2_\rig(X/\rk)$ if $\ell=p$ and $\mathrm{char}(F)=p$;
\end{itemize}
then in all cases we get a monodromy operator $N$ on $H^2(X)$.

\begin{theorem} \label{k3main} Let $\cur{X}/R$ be a minimal semistable model of a K3 surface $X$, and $Y$ its special fibre, which is a combinatorial K3 surface. Then $Y$ is of Type I,II or III respectively as the nilpotency index of $N$ on $H^2(X)$ is $1$, $2$ or $3$.
\end{theorem}

\begin{proof} The case $\ell=p$ and $\mathrm{char}(F)=0$ is due to Hern\'andez Mada, and in fact the case $\ell=\mathrm{char}(F)=p$ also follows from his result by applying the results in Chapter 5 of \cite{LP16}.

To deal with the case $\ell\neq \mathrm{char}(k)$, we use the weight spectral sequence (for algebraic spaces this is Proposition 2.3 of \cite{Kul77}). Let $Y=Y_1\cup \ldots \cup Y_N$ be the components of $Y$, $C_{ij}=Y_i\cap Y_j$ the double curves and
\[ Y^{(0)}=\CMcoprod_i Y_i,\;\;Y^{(1)}=\CMcoprod_{i<j}C_{ij},\;\;Y^{(2)}=\CMcoprod_{i<j<k}Y_i\cap Y_j\cap Y_k.\]
We consider the weight spectral sequence
\[ E_1^{s,t} = \bigoplus_{j\geq \max\{0,-s\}} H^{t-2j}_\et(Y^{(s+2j)}_{\overline{k}},\Q_\ell)(-j)\Rightarrow H^{s+t}_\et(X_{\overline{F}},\Q_\ell) \]
which degenerates at $E_2$ and is compatible with monodromy in the sense that there exists a morphism $N:E_r^{s,t}\rightarrow E_r^{s+2,t-2}$ of spectral sequences abutting to the monodromy operator on $H^{s+t}_\et(X_{\overline{K}},\Q_\ell)$. Moreover, by the weight-monodromy conjecture (see Remark 6.8(1) of \cite{Nak06}) we know that $N^r$ induces an isomorphism $E_2^{-r,w+r}\isomto E_2^{r,w-r}$. Hence we can characterise the three cases where $N$ has nilpotency index $1,2$ or $3$ in terms of the weight spectral sequence as follows.
\begin{enumerate} \item $N=0$ if and only if $E_2^{1,1}=E_2^{2,0}=0$. 
\item $N\neq 0$, $N^2=0$ if and only if $E_2^{1,1}\neq0$ and $E_2^{2,0}=0$.
\item $N^2\neq0$, $N^3=0$ if and only if $E_2^{1,1},E_2^{2,0}\neq 0$.
\end{enumerate}
Hence it suffices to show the following.
\begin{enumerate} \item If $Y$ is of Type I, then $E_2^{1,1}=0$. 
\item If $Y$ is of Type II, then $E_2^{1,1}\neq0$ and $E_2^{2,0}=0$.
\item If $Y$ is of Type III, then $E_2^{2,0}\neq 0$.
\end{enumerate}
The first of these is clear, and in both the Type II and III cases the term
\[E_2^{2,0}= \mathrm{coker}\left( H^0(Y^{(1)}_{\overline{k}},\Q_\ell) \rightarrow H^0(Y^{(2)}_{\overline{k}},\Q_\ell) \right) \]
is simply the second singular cohomology $H^2_\mathrm{sing}(\Gamma,\Q_\ell)$ of the dual graph $\Gamma$. For Type II this is $0$, and for Type III this is $1$-dimensional over $\Q_\ell$, hence it suffices to show that if $Y$ is of Type II, then $E_2^{1,1}\neq0$.

But we know that
\[ \dim_{\Q_\ell} E_2^{-1,2}+\dim_{\Q_\ell} E_2^{0,1}+\dim_{\Q_\ell} E_2^{1,0}=\dim_{\Q_\ell} H^1_\et(X,\Q_\ell)=0  \]
and hence $\dim_{\Q_\ell}E_2^{0,1}=0$. Therefore we have
\[ \dim_{\Q_\ell} E_2^{1,1}= \dim_{\Q_\ell}H^1_\et(Y^{(1)},\Q_\ell)-\dim_{\Q_\ell} H^1_\et(Y^{(0)},\Q_\ell) \]
which using Lemma \ref{betti} we can check to be equal to $2(N-1)-2(N-2)=2$. Hence $E_2^{1,1}\neq0$ as required.
\end{proof}

\section{Enriques surfaces}\label{en}

To deal with the case of Enriques surfaces, we again start with a result of Nakkajima, analogous to the one quoted above. 

\begin{theorem}[\cite{Nak00}, \S7] Let $Y^{\log}$ be a logarithmic Enriques surface over $k$. Then the underlying scheme $Y$ is a combinatorial Enriques surface.
\end{theorem}

\begin{remark} Again, it is possible to prove this only using coherent cohomology as in Theorem \ref{hst} below.
\end{remark}

\begin{corollary} Let $\mathscr{X}/R$ be a minimal semistable model of an Enriques surface $X/F$. Then the special fibre $Y$ is a combinatorial Enriques surface.
\end{corollary}

If $X/F$ is an Enriques surface, then for all $\ell\neq p$ the second homotopy group $\pi_2^\et(X_{\overline{F}})_{\Q_\ell}$ (for the definition of the higher homotopy groups of algebraic varieties, see \cite{AM69}) is a finite dimensional $\Q_\ell$ vector space with a continuous Galois action, which is quasi-unipotent. If $\ell=p$ and $\mathrm{char}(F)=0$ then $\pi_2^\et(X_{\overline{F}})_{\Q_p}$ is a de Rham representation of $G_F$. If $\mathrm{char}(F)=p$ there is (currently!) no general theory of higher homotopy groups, so instead we cheat somewhat and use the known properties of the higher \'etale homotopy groups to justify making the following definition.

\begin{definition} We define $\pi_2^\rig(X/\rk):=H^2_\rig(\widetilde{X}/\rk)^\vee$, where $\widetilde{X}\rightarrow X$ is the canonical double cover of $X$.
\end{definition}

Thus $\pi_2^\rig(X/\rk)$ is a $\pn$-module over $\rk$. Again, if we let $\pi_2(X)$ stand for any of $\pi_2^\et(X_{\overline{F}})_{\Q_\ell}$, $\pi_2^\et(X_{\overline{F}})_{\Q_p}$ or $\pi_2^\rig(X/\rk)$, then in all cases we have a monodromy operator $N$ associated to $\pi_2(X)$.

\begin{theorem} \label{cgre} Let $\cur{X}/R$ be a minimal semistable model of an Enriques surface $X$, and $Y$ its special fibre, which is a combinatorial Enriques surface. Then $Y$ is of Type I,II or III respectively as the nilpotency index of $N$ on $\pi_2(X)$ is $1$, $2$ or $3$.
\end{theorem}

\begin{remark} \begin{enumerate} \item As noted in the introduction, a result very similar to this was proved in \cite{HM15}.
\item The result as stated here is slightly different to Theorem \ref{vague}. There are in fact two ways of stating it, one using the second homotopy group $\pi_2$ and one using the cohomology of a certain rank 2 local system $V$ on $X$, given by pushing forward the constant sheaf on the K3 double cover of $X$.
\end{enumerate}
\end{remark}

\begin{proof} If we let $\widetilde{\cur{X}}$ denote the canonical double cover of $\cur{X}$, with special fibre $\widetilde{Y}$ and generic fibre $\widetilde{X}$, then as remarked above, $\widetilde{X}$ is a smooth K3 surface over $K$, and $\widetilde{\cur{X}}$ is a minimal semistable model for $\widetilde{X}$. Hence $\widetilde{Y}$ is a combinatorial K3 surface, whose type can be deduced from the nilpotency index of the monodromy operator $N$ on $H^2_\et(\widetilde{X}_{\overline{F}},\Q_\ell)$. 

Now note that since $\widetilde{X}$ is simply connected, we have 
\[ \pi_2^\et(X_{\overline{F}})_{\Q_\ell} \cong \pi_2^\et(\widetilde{X}_{\overline{F}})_{\Q_\ell}  \cong H_2^\et(\widetilde{X}_{\overline{F}},\Q_\ell) \cong H^2_\et(\widetilde{X}_{\overline{F}},\Q_\ell)^\vee \]
for all $\ell$ (including $\ell=p$ when $\mathrm{char}(F)=0$), and the corresponding isomorphism holds by definition for $\pi_2^\rig(X/\rk)$. Hence $\widetilde{Y}$ is of Type I,II or III respectively as the nilpotency index of $N$ on $\pi_2(X)$ is 1,2 or 3. It therefore suffices to show that the type of $\widetilde{Y}$ is the same as that of $Y$. 

Note that we have a finite \'{e}tale map $f:\widetilde{Y}\rightarrow Y$, therefore if $\widetilde{Y}$ is of Type I, that is a smooth K3 surface, then we must also have that $Y$ is smooth, hence of Type I. If $Y$ is not smooth, then let the components of $Y$ be $Y_1,\ldots,Y_N$, and the components of $\widetilde{Y}$ be $\widetilde{Y}_1,\ldots,\widetilde{Y}_M$. After pulling back $f$ to each component $Y_i$, one of two things can occur:
\begin{enumerate} \item $f^{-1}(Y_i)$ is irreducible, and we get a non-trivial 2-cover $\widetilde{Y}_j\rightarrow Y_i$;
\item $f^{-1}(Y_i)$ splits into 2 disjoint components $\widetilde{Y}_j,\widetilde{Y}_{j'}$, each mapping isomorphically onto $Y$.
\end{enumerate}
If $\widetilde{Y}$ is of Type III, then each component $\widetilde{Y}_j$ is rational, hence, since rational varieties are simply connected each component of $Y$ is also rational, and $Y$ is of Type III. If $\widetilde{Y}$ is of Type II, then one of two things can happen.
\begin{enumerate}
\item $M>2$ and there exists a component of $\widetilde{Y}$ which is an elliptic ruled surface.
\item $M=2$ and $\widetilde{Y}=\widetilde{Y}_1\cup\widetilde{Y}_2$ consists of $2$ rational surfaces meeting along an elliptic curve.
\end{enumerate}
In the first case, one verifies that $Y$ must also have a component isomorphic to an elliptic ruled surface (since a rational surface cannot be an unramified cover of an elliptic ruled surface), and is therefore of Type II. In the second case, $Y$ must also have 2 components, (since otherwise $Y$, and therefore $\widetilde{Y}$, would be smooth), and each component of $\widetilde{Y}$ would be a non-trivial double cover of a component of $Y$. But since the components of $Y$ are either rational or elliptic ruled, this cannot happen.
\end{proof}

\section{Abelian surfaces}

In order to deal with abelian surfaces, we need the following analogue of Nakkajima's result,

\begin{theorem}\label{crass} Let $Y^{\log}$ be a logarithmic abelian surface over $k$. Then the underlying scheme $Y$ is a combinatorial abelian surface.
\end{theorem}

\begin{proof} We may assume that $k=\overline{k}$. We adapt the proof of Theorem II of \cite{Kul77}. Let $Y_1,\ldots,Y_N$ denote the components of $Y$, $C_{ij}=Y_i\cap Y_j$ for $i\neq j$ the double curves, and $T_{C_{ij}}$ the number of triple points on each curve $C_{ij}$. We may assume that $N>1$. 

Note that $\omega_Y|_{Y_i} \cong \Omega^2_{Y_i/k}(\log \sum_{j\neq i} C_{ij})\cong \cur{O}_{Y_i}$ and hence the divisor $K_{Y_i} + \sum_{j\neq i} C_{ij}$ on $Y_i$ is principal, where $K_{Y_i}$ is a canonical divisor on $Y_i$. Write $D_i=\sum_{j\neq i} C_{ij}$. Now applying Lemma \ref{paac} gives us the following possibilities for each $(Y_i,D_i)$:
\begin{enumerate} \item $Y_i$ is an elliptic ruled surface, and either:
\begin{enumerate} \item  $D_i=E_1+E_2$ where $E_1,E_2$ are 2 non-intersecting rulings;
\item a $D_i=E$ is a single 2-ruling.
\end{enumerate}
\item $Y_i$ is a rational surface, and either:
\begin{enumerate} \item $D_i=E$ is an elliptic curve inside $Y_i$;
\item $D_i=C_1+\ldots +C_d$ is a cycle of rational curves on $Y_i$.
\end{enumerate}
\end{enumerate}
First suppose that there is some $i$ such that case (2)(b) happens. Then this must also occur on each neighbour of $Y_i$, and since $Y$ is connected, it follows that this occurs on each component. The dual graph $\Gamma$ is therefore a triangulation of a compact surface without border.

Write
\[ Y^{(0)}=\CMcoprod_i Y_i,\;\;Y^{(1)}=\CMcoprod_{i<j}C_{ij},\;\;Y^{(2)}=\CMcoprod_{i<j<k}Y_i\cap Y_j\cap Y_k,\]
and consider the spectral sequence $H^t(Y^{(s)},\cur{O}_{Y^{(s)}})\Rightarrow H^{s+t}(Y,\cur{O}_Y) $ constructed in \S\ref{snclv}. Since the components $Y_i$ and the curves $C_{ij}$ are rational, it follows that $H^t(Y^{(s)},\cur{O}_{Y^{(s)}})=0$ for $t>0$ (see for example, Theorem 1 of \cite{CR11}), and therefore that the coherent cohomology $H^i(Y,\cur{O}_Y)$ of $Y$ is the same as the $k$-valued singular cohomology $H^i_\mathrm{sing}(\Gamma,k)$ of $\Gamma$. But since $p\neq 2$, the $k$-Betti numbers $\dim_k H^i_\mathrm{sing}(\Gamma,k)$ are the same as the $\Q$-Betti numbers $\dim_{\Q}H^i_\mathrm{sing}(\Gamma,\Q)$, the latter must therefore be $1,2,1$ and by the classification of closed 2-manifolds we can deduce that $\Gamma$ is a torus.

Finally let us suppose that all the double curves $C_{ij}$ are elliptic curves, so that each $T_{C_{ij}}=0$ (see the proof of Lemma \ref{paac}). Again examining the spectral sequence $H^t(Y^{(s)},\cur{O}_{Y^{(s)}})\Rightarrow H^{s+t}(Y,\cur{O}_Y)$ and using the fact that $\chi(E)=0$ for an elliptic curve, we can see that $0=\chi(Y)=\chi(Y^{(0)})=\bigoplus_i\chi(Y_i)$. Since each $Y_i$ is either rational ($\chi=1$) or elliptic ruled ($\chi=0$), it follows that each $Y_i$ must be elliptic ruled, and we are in the case (1) above. The dual graph $\Gamma$ is one dimensional, and since each component has on it at most two double curves, $\Gamma$ is either a line segment or a circle.

If $\Gamma$ were a line segment, then $Y=Y_1\cup_{E_1}\ldots \cup_{E_{N-1}} Y_N$ would be a chain. Then birational invariance of coherent cohomology would imply that the maps
\begin{align*} H^0(Y_i,\mathcal{O}_{Y_i}) &\rightarrow H^0(E_i,\mathcal{O}_{E_i}) \\
H^0(Y_{i+1},\mathcal{O}_{Y_{i+1}}) &\rightarrow H^0(E_i,\mathcal{O}_{E_i}) \\ 
H^1(Y_i,\mathcal{O}_{Y_i}) &\rightarrow H^1(E_i,\mathcal{O}_{E_i}) \\
H^1(Y_{i+1},\mathcal{O}_{Y_{i+1}}) &\rightarrow H^1(E_i,\mathcal{O}_{E_i}) 
\end{align*}
would be isomorphisms, and hence some basic linear algebra would imply surjectivity of the maps
\begin{align*} H^0(Y^{(0)},\mathcal{O}_{Y^{(0)}}) &\rightarrow H^0(Y^{(1)},\mathcal{O}_{Y^{(1)}}) \\
H^1(Y^{(0)},\mathcal{O}_{Y^{(0)}}) &\rightarrow H^1(Y^{(1)},\mathcal{O}_{Y^{(1)}}).
\end{align*}
Also, we would have $\dim_k H^1(Y^{(0)},\mathcal{O}_{Y^{(0)}})={N}$ and $\dim_k H^1(Y^{(1)},\mathcal{O}_{Y^{(1)}})={N-1}$, and hence again examining the spectral sequence $H^t(Y^{(s)},\cur{O}_{Y^{(s)}})\Rightarrow H^{s+t}(Y,\cur{O}_Y)$ would imply that $\dim_k H^1(Y,\mathcal{O}_Y)=1$. Since we know that in fact $\dim_kH^1(Y,\mathcal{O}_Y)=2$ (by the definition of a logarithmic abelian surface), this cannot happen. Hence $\Gamma$ must be a circle and $Y$ is of Type II.
\end{proof}

\begin{corollary} Let $\mathscr{X}/R$ be a minimal semistable model of an abelian surface $X/F$. Then the special fibre $Y$ is a combinatorial abelian surface.
\end{corollary}

If $X/F$ is an abelian surface, then for any prime $\ell\neq p$ we consider the quasi-unipotent $G_F$-representation $H^2_\et(X_{\overline{K}},\Q_\ell)$. For $\ell= p$ and $\mathrm{char}(F)=0$ we may also consider the de Rham representation $H^2_\et(X_{\overline{K}},\Q_p)$, and when $\mathrm{char}(F)=p=\ell$ the $\pn$-module $H^2_\rig(X/\rk)$. Again letting $H^2(X)$ stand for any of the above second cohomology groups then, in each case, we have a nilpotent monodromy operator $N$ associated to $H^2(X)$.

\begin{theorem} \label{abmain} Let $\cur{X}/R$ be a minimal semistable model for $X$, with special fibre $Y$. Then $Y$ is combinatorial of Type I, II or III respectively as the nilpotency index of $N$ on $H^2(X)$ is $1$, $2$ or $3$.
\end{theorem}

\begin{proof}
We will treat the case $\ell\neq p$ and $\mathrm{char}(F)=0$, the other cases are handled entirely similarly. Let $Y_1,\ldots , Y_N$ be the smooth components of the special fibre $Y$. For any $I=\{ i_1\ldots,i_n\}$ write $Y_I=\cap_{i\in I} Y_i$ and for any $s\geq 0$ write $Y^{(s)}=\CMcoprod_{\norm{I}=s+1}Y_I$, these are all smooth over $k$ and empty if $s>2$. 

As in the proof of Theorem \ref{k3main} we consider the weight spectral sequence
\[ E_1^{s,t} = \bigoplus_{j\geq \max\{0,-s\}} H^{t-2j}_\et(Y^{(s+2j)}_{\overline{k}},\Q_\ell)(-j)\Rightarrow H^{s+t}_\et(X_{\overline{F}},\Q_\ell). \]
As before it suffices to show the following.
\begin{enumerate} \item If $Y$ is of Type I, then $E_2^{1,1}=0$. 
\item If $Y$ is of Type II, then $E_2^{1,1}\neq0$ and $E_2^{2,0}=0$.
\item If $Y$ is of Type III, then $E_2^{2,0}\neq 0$.
\end{enumerate}
Again, the first of these is trivial, and in both the Type II and III cases the term $E_2^{2,0}$ is the second singular cohomology $H^2_\mathrm{sing}(\Gamma,\Q_\ell)$ of the dual graph $\Gamma$. It therefore suffices to show that if $Y$ is of Type II, then $E_2^{1,1}\neq0$.

To show this, note that we have $\dim_{\Q_\ell} E_2^{i,0}=\dim_{\Q_\ell} H^i_\mathrm{sing}(\Gamma,\Q_\ell)$, which is $1$ for $i=0,1$ and zero otherwise. Hence by the fact that $E_2^{-r,w+r}\isomto E_2^{r,w-r}$ we may deduce that $\dim_{\Q_\ell}E_2^{-1,2}=1$, and hence from the fact that
\[ \dim_{\Q_\ell}E_2^{-1,2}+\dim_{\Q_\ell}E_2^{0,1}+\dim_{\Q_\ell}E_2^{1,0}=\dim_{\Q_\ell} H^1_\et(X_{\overline{F}},\Q_\ell)=4 \]
that $\dim_{\Q_\ell}E_2^{0,1}=2$. If we write $Y=Y_1\cup \ldots \cup Y_N$ as a union of $N$ elliptic ruled surfaces, then $Y^{(1)}$ is a disjoint union of $N$ elliptic curves. Hence by Lemma \ref{betti} we must have 
\[ \dim_{\Q_\ell}H^1_\et(Y^{(0)}_{\overline{k}},\Q_\ell)= \dim_{\Q_\ell}H^1_\et(Y^{(1)}_{\overline{k}},\Q_\ell)=2N.\]
Hence $\dim_{\Q_\ell}E_2^{1,1}=\dim_{\Q_\ell}E_2^{0,1}=2$ and therefore $E_2^{1,1}\neq0$.

When $\ell=p$, the $\ell$-adic weight spectral should be replaced by the $p$-adic one constructed by Mokrane in \cite{Mok93}. That this abuts to the $p$-adic \'{e}tale cohomology when $\mathrm{char}(F)=0$ follows from Matsumoto's extension of Fontaine's $C_{\mathrm{st}}$ conjecture to algebraic spaces in \cite{Mat15}, and that it abuts to the $\rk$-valued rigid cohomology when $\mathrm{char}(F)=p$ follows from Proposition \ref{cstpas}.
\end{proof}

\section{Bielliptic surfaces}

We can now complete our treatment of minimal models of surfaces of Kodaira dimension $0$ by investigating what happens for bielliptic surfaces.

\begin{theorem}\label{hst} Let $Y^{\log}$ be a logarithmic bielliptic surface over $k$. Then the underlying scheme $Y$ is a combinatorial bielliptic surface. \end{theorem}

\begin{proof} We may assume $k=\overline{k}$. Let $\pi:\widetilde{Y}^{\log}\rightarrow Y^{\log}$ be the canonical $m$-cover associated to $\omega_{Y^{\log}}$. Then one easily checks that $\widetilde{Y}^{\log}$ is a logarithmic abelian surface over $k$, and hence is combinatorial of Type I, II or III. If $\widetilde{Y}$ is of Type I, then $\widetilde{Y}$, and therefore $Y$, must be smooth over $k$, and hence $Y$ is a smooth bielliptic surface over $k$, i.e. of Type I. 

So assume that $\widetilde{Y}$ is of Type II or III. Let $\widetilde{Y}_1,\ldots,\widetilde{Y}_M$ denote the components of $\widetilde{Y}$ and $Y_1,\ldots,Y_N$ those of $Y$. Note that as in the proof of Theorem \ref{crass} we have
\[ m(K_{Y_i}+\sum_{j\neq i}C_{ij})=0 \]
in $\mathrm{Pic}(Y_i)$, where $C_{ij}$ are the double curves. 

Suppose that $\widetilde{Y}$ is of Type II. Note that each component of $\widetilde{Y}$ is finite \'{e}tale over some component of $Y$, and hence each component of $Y$ is an elliptic ruled surface. For each $Y_i$ choose some $\widetilde{Y}_l \rightarrow Y_i$ finite \'etale, and let $\widetilde{C}_{lj}$ be the inverse image of the double curves. Then we have 
\[ K_{\widetilde{Y}_l}+\sum_j\widetilde{C}_{lj}=0\]
in $\mathrm{Pic}(\widetilde{Y}_l)$. Applying Lemma \ref{paac} we can see that $\sum_j\widetilde{C}_{il}$ is either a single elliptic curve $E$, which is a 2-ruling on $\widetilde{Y}_l$, or two disjoint rulings $E_1,E_2$. Hence the same is true for $\sum_j C_{ij}$ on $Y_i$, and therefore $Y$ is of Type II.

Finally, suppose that $\widetilde{Y}$ is of Type III. Then again, each component of $\widetilde{Y}$ is finite \'{e}tale over some component of $Y$, hence all of the latter are rational. Since the Picard group of a rational surface is torsion free, it follows that we must have
\[ K_{Y_i}+\sum_{j\neq i}C_{ij} =0\]
on each $Y_i$. Hence applying Lemma \ref{paac} as in the proof of Theorem \ref{crass} it suffices to show that the dual graph $\Gamma$ of $Y$ is a triangulation of the Klein bottle. But now examining the spectral sequence 
\[ E^{s,t}_1:=H^t(Y^{(s)},\cur{O}_{Y^{(s)}})\Rightarrow H^{s+t}(Y,\cur{O}_Y) \]
(where $Y^{(s)}$ is defined similarly to before), and using the fact that $\mathrm{char}(k)>2$, we can see that the Betti numbers of $\Gamma$ are the same as the dimensions of the coherent cohomology of $Y$, and therefore $Y$ is of Type III. \end{proof}

To formulate the analogue of Theorem \ref{abmain} for bielliptic surfaces, we will need to construct a family of canonical local systems on our bielliptic surface $X$. Note that the torsion element $\omega_X\in \mathrm{Pic}(X)[m]\in H^1(X,\mu_m)$ gives rise to a $\mu_m$-torsor over $X$, and hence a canonical $\Q$-valued permutation representation $\rho$ of the fundamental group $\pi_1^\et(X,\bar{x})$, and we can use this to construct canonical $\ell$- or $p$-adic local systems on $X$. When $\ell\neq p$ we obtain a continuous representation $\rho\otimes_{\Q} \Q_\ell$ of $\pi_1^{\et}(X,\bar{x})$ and hence a lisse $\ell$-adic sheaf $V_\ell$ on $X$, and when $\ell=p$ and $\mathrm{char}(F)=0$ we may do the same to obtain a lisse $p$-adic sheaf $V_p$ on $X$, and when $\ell=\mathrm{char}(F)=p$ we obtain an overconvergent $F$-isocrystal $V_p$ on $X/K$ using the construction of \S\ref{rpad}.

Then the local systems $V_\ell,V_p$ do not depend on the choice of point $\bar{x}$, and the $G_F$-representations $H^2_\et(X_{\overline{F}},V_\ell)$ and $H^2_\et(X_{\overline{F}},V_p)$ when $\mathrm{char}(F)=0$ are quasi-unipotent and de Rham respectively, we may also consider the $\pn$-module
\[ H^2_\rig(X/\rk,V_p):=H^2_\rig(X/\ekd,V_p)\otimes_{\ekd}\rk\]
over $\rk$. Letting $H^2(X,V)$ stand for any of $H^2_\et(X_{\overline{F}},V_\ell)$, $H^2_\et(X_{\overline{F}},V_p)$ or $H^2_\et(X/\cur{R}_K,V_p)$, in all cases we obtain monodromy operators $N$ associated to $H^2(X,V)$. 

\begin{remark} This construction might seem a little laboured, since what we are really constructing is simply the pushforward of the constant sheaf via the canonical abelian cover of $X$. The point of describing it in the above way is to emphasise the fact that the local systems $V_\ell,V_p$ are entirely intrinsic to $X$.
\end{remark}

\begin{theorem} \label{bimain} Let $\cur{X}/R$ be a minimal semistable model for $X$, with special fibre $Y$. Then $Y$ is combinatorial of Type I, II or III respectively as the nilpotency index of $N$ on $H^2(X,V)$ is $1$, $2$ or $3$.
\end{theorem}

\begin{proof} The local systems $V_\ell,V_p$ are by construction such that there exists a finite \'{e}tale cover $\widetilde{\cur{X}}\rightarrow \cur{X}$ Galois with group $G$, such that $\widetilde{\cur{X}}$ is a minimal model of an abelian surface $\widetilde{X}$ and $H^2(X,V)\cong H^2(\widetilde{X})$. The special fibre $\widetilde{Y}$ is therefore a finite \'{e}tale cover of $Y$, also Galois with group $G$, and is a combinatorial abelian surface of Type I, II or III according to the nilpotency index of $N$ on $H^2(X,V)$. Hence we must show that $\widetilde{Y}$ and $Y$ have the same type; this was shown during the course of the proof of Theorem \ref{hst}.
\end{proof}

\section{Existence of models and abstract good reduction}

As explained in the introduction, our results so far are essentially `one half' of the good reduction problem for surfaces with $\kappa=0$, the other half consists of trying to actually find models nice enough to be able to apply the above methods.

\begin{definition} Let $X/F$ be a minimal surface of Kodaira dimension $0$. Then we say that $X$ admits potentially combinatorial reduction if after replacing $F$ by a finite separable extension, there exists a minimal model $\mathcal{X}/R$ of $X$.
\end{definition}

Then thanks to the results of the previous sections, for surfaces with potentially combinatorial reduction, we can describe the `type' of the reduction in terms of the nilpotency index of the monodromy operator on a suitable cohomology or homotopy group of $X$ (either $\ell$-adic or $p$-adic). We can therefore answer questions of `abstract reduction' type by establishing whether or not surfaces have potentially combinatorial reduction. The strongest result one might hope for is that every such surface has potentially combinatorial reduction. Unfortunately, this is not the case.

\begin{example}[\cite{LM14}, Theorem 2.8] There exist Enriques surfaces over $\Q_p$ which do not admit potentially combinatorial reduction.
\end{example}

This can in fact be seen already in the complex analytic case of a degenerating family $\mathscr{X}\rightarrow \Delta$ of K\"ahler manifolds over a disc (see Appendix 2 of \cite{Per77}). In Proposition 2.1 of \cite{LM14} it is shown that if a K3 surface over $F$ admits potentially strictly semistable reduction, then it admits potentially combinatorial reduction. Again, while the former is always conjectured, it can only be proved under certain conditions, see Corollary 2.2 of \emph{loc. cit}. Since we know that abelian surfaces admit potentially strictly semistable reduction, we can use their argument to prove the following.

\begin{theorem} \label{concab} Abelian surfaces $X/F$ admit potentially combinatorial reduction.
\end{theorem}

\begin{proof} By Theorem 4.6 of \cite{Kun98}, after replacing $F$ by a finite separable extension, we may assume that there exists a strictly semistable scheme model $\mathcal{X}/R$ of $X$. By applying the Minimal Model Program of \cite{Kaw94} there exists another scheme model $\mathcal{X}'$ for $X$ such that:
\begin{enumerate}
\item the components of the special fibre of $\mathcal{X}'$ are geometrically normal and integral $\Q$-Cartier divisors on $\mathcal{X}'$;
\item $\mathscr{X}'$ is regular away from a finite set $\Sigma$ of closed points on its special fibre, and $\mathcal{X}'$ has only terminal singularities at these points;
\item the special fibre is a normal crossings divisor on $\mathcal{X}\setminus \Sigma$;
\item the relative canonical Weil divisor $K_{\mathcal{X}'/R}$ is $\Q$-Cartier and n.e.f. relative to $R$.
\end{enumerate}
Now, since the canonical divisor $K_{X}$ on the generic fibre is trivial, it follows that we may write $K_{\mathscr{X}'/R}$ as a linear combination $\sum_i a_iV_i$ of the components of the special fibre $Y'$ of $\mathcal{X}'$. Moreover since $\sum_i V_i=0$ we may in fact assume that $a_i\leq 0$ for all $i$ and $a_i=0$ for some $i$. Since $K_{\mathscr{X}'/\Q}$ is n.e.f. relative to $R$, arguing as in Lemma 4.7 of \cite{Mau14} shows that in fact we must have $a_i=0$ for all $i$, and hence $K_{\mathscr{X}'/R}=0$. In particular it is Cartier (not just $\Q$-Cartier) and therefore applying Theorem 4.4 of \cite{Kaw94} we can see that in fact $\mathcal{X}'$ is strictly semistable away from a finite set of isolated rational double points on components of $Y'$. 

Finally, applying Theorem 2.9.2 of \cite{Sai04} and Theorem 2 of \cite{Art74} we may, after replacing $F$ by a finite separable extension, find a strictly semistable algebraic space model $\mathscr{X}''/R$ for $X$ and a birational morphism $\mathscr{X}''\rightarrow \mathscr{X}'$ which is an isomorphism outside a closed subset of each special fibre, of codimension $\geq 2$ in the total space. Since we know that $K_{\mathscr{X}'/R}=0$, it follows that $K_{\mathscr{X}''/R}=0$, and therefore $\mathscr{X}''$ is a minimal model in the sense of Definition \ref{minmod}.
\end{proof}

\begin{remark} Of course, this begs the question as to whether or not bielliptic surfaces admit potentially combinatorial reduction, we are not sure whether to expect this or not.
\end{remark}

Finally, we would like to relate the `type' of combinatorial reduction for abelian (and hence bielliptic) surfaces to the more traditional invariants associated to abelian varieties with semi-abelian reduction. So suppose that we have an abelian surface $X/F$. Then after a finite separable extension, we may assume that $X$ admits the structure of an abelian variety over $F$, let us therefore call it $A$ instead. After making a further extension, we may assume that $A$ has semi-abelian reduction, i.e. there exist a semi-abelian scheme over $R$ whose generic fibre is $A$. In this situation we have a `uniformisation cross' for $A$ (see for example \S2 of \cite{CI99}), which is a diagram
\[ \xymatrix{ & T\ar[d] & \\ \Gamma\ar[r] & G\ar[r]^{\pi} \ar[d] & A \\ & B } \]
where $T$ is a torus over $F$, $B$ is an abelian variety with good reduction, $G$ is an extension of $B$ by $T$ and $\Gamma$ is a discrete group. Fixing a prime $\ell\neq p$, the monodromy operator on $H^1_\et(A_{\overline{F}},\Q_\ell)$ can be defined as follows. We have an exact sequence
\[ 0\rightarrow \mathrm{Hom}(\Gamma,\Q_\ell)\rightarrow H^1_\et(A_{\overline{F}},\Q_\ell)\rightarrow H^1_\et(G_{\overline{F}},\Q_\ell)\rightarrow 0\]
and a non-degenerate pairing 
\[ \Gamma\times \mathrm{Hom}(T,\G_m)\rightarrow \Q \]
and the monodromy operator on $H^1_\et(A_{\overline{F}},\Q_\ell)$ is then the composition
\[ H^1_\et(A_{\overline{F}},\Q_\ell)\rightarrow H^1_\et(T_{\overline{F}},\Q_\ell)\rightarrow \mathrm{Hom}(T,\G_m)\otimes_{\Z}\Q_\ell\rightarrow \mathrm{Hom}(\Gamma,\Q_\ell)\rightarrow H^1_\et(A_{\overline{F}},\Q_\ell)\]
(see for example \cite{CI99}). Since the first map is surjective, the last injective, and all others are isomorphisms, we have that the dimension of the image of monodromy on $H^1_\et(A_{\overline{F}},\Q_\ell)$ is equal to the dimension of $H^1_\et(T_{\overline{F}},\Q_\ell)$, and therefore to the rank of $T$. Using some simple linear algebra, one can therefore give the nilpotency index of $N$ on $H^2_\et(A_{\overline{F}},\Q_\ell)=\bigwedge^2  H^1_\et(A_{\overline{F}},\Q_\ell)$ as follows:
\begin{enumerate}\item $\mathrm{rank}(T)=0\Rightarrow N=0$ on $ H^2_\et(A_{\overline{F}},\Q_\ell)$;
\item $\mathrm{rank}(T)=1 \Rightarrow N\neq 0$, $N^2=0$ on $ H^2_\et(A_{\overline{F}},\Q_\ell)$;
\item $\mathrm{rank}(T)=2 \Rightarrow N^2\neq 0$, $N^3=0$ on $ H^2_\et(A_{\overline{F}},\Q_\ell)$.
\end{enumerate}
Hence we have the following.
\begin{proposition} \label{typerank} $A$ has potentially combinatorial reduction of Type I,II or III as $\mathrm{rank}(T)$ is $0,1$ or $2$ respectively.
\end{proposition}

\section{Towards higher dimensions}

In this final section of the article, we begin to investigate the shape of degenerations in higher dimensions, in particular looking at Calabi--Yau threefolds and concentrating on the `maximal intersection case', analogous to the Type III degeneration of K3 surfaces. In characteristic 0 some fairly general results in this direction are proved in \cite{KX15}, and the approach there provides much of the inspiration for the main result of this section, Theorem \ref{CY3}, as well as some of the key ingredients of its proof. Many of the proofs there rely on results from the log minimal model program (LMMP), which happily has recently been solved for threefolds in characteristics $>5$ by Hacon--Xu \cite{HX15}, Birkar \cite{Bir13} and Birkar--Waldron \cite{BW14}. Given these results, many of our proofs consist of working through special low dimensional cases of \cite{KX15} explicitly (and gaining slightly more information than given there), although there are certain places where specifically characteristic $p$ arguments are needed.

Since we will need to use the LMMP for threefolds, we will assume throughout that $p>5$. Unfortunately, since we will also need to know results on the homotopy type of Berkovich spaces, we will also need to assume that our models are in fact schemes, rather than algebraic spaces.

\begin{definition} A Calabi--Yau variety over $F$ is a smooth, projective, geometrically connected variety $X/F$ such that:
\begin{itemize}\item the canonical sheaf $\omega_X=\Omega^{\dim X}_{X/F}$ is trivial, i.e. $\omega_X\cong \mathcal{O}_X$; 
\item $X$ is geometrically simply connected, i.e. $\pi_1^{\et}(X_{\overline{F}},x)=\{1\}$ for any $x\in X(\overline{F})$;
\item $H^i(X,\mathcal{O}_X)=0$ for all $0<i<\dim X$.
\end{itemize}
\end{definition}

In dimension 2 these are exactly the K3 surfaces, and we will be interested in what we can say about degenerations of Calabi--Yau varieties in dimension 3. Here one expects to be able to divide `suitably nice' semistable degenerations into 4 `types' depending on the nilpotency index of $N$ acting on $H^3(X)$ (for some suitable Weil cohomology theory). In this section we will treat the `Type IV' situation.

\begin{definition} We say that a morphism $f:X\rightarrow S$ of algebraic varieties (over an algebraically closed field) is a Mori fibre space if it is projective with connected fibres, and the anticanonical divisor $-K_X$ is $f$-ample, i.e. ample on all fibres of $f$.
\end{definition}

\begin{definition} \label{ccyt4} Let $Y=\cup_i V_i$ be a simple normal crossings variety over $k$ of dimension 3. We say that $Y$ is a combinatorial Calabi--Yau of Type IV if geometrically (i.e. over $\overline{k}$) we have:
\begin{itemize}  \item each component $V_i$ is birational to a Mori fibre space over a unirational base;
\item  each connected component of every double surface $S_{ij}$ is rational;
\item each connected component of every triple curve $C_{ijk}$ is rational;
\item the dual graph $\Gamma$ of $Y$ is a triangulation of the 3-sphere $S^3$.
\end{itemize}
\end{definition}

\begin{remark} \begin{enumerate} \item It is worth noting that in characteristic $0$ these conditions imply that $V_i$ is rationally connected, and the analogue of the condition in dimension $2$ implies rationality, even in characteristic $p$.
\item We may in fact assume that we have the above shape after a finite extension of $k$.
\end{enumerate}
\end{remark}

Let $H^3(X)$ stand for either $H^3_\et(X_{\overline{F}},\Q_\ell)$ if $\mathrm{char}(F)=0$ or $\ell \neq p$, or $H^3_\rig(X/\rk)$ if $\mathrm{char}(F)=p$. In all cases, we have a natural monodromy operator $N$ acting on $H^3(X)$, such that $N^4=0$. As a first step in the study of Calabi--Yau degenerations in dimension 3, the main result of this section is the following.

\begin{theorem} \label{CY3} Let $\mathscr{X}$ be a strictly semistable $R$-scheme with generic fibre $X$ a Calabi--Yau threefold. Assume moreover that the sheaf of logarithmic $3$-forms $\omega_\mathscr{X}$ on $\mathscr{X}$ relative to $R$ is trivial, and that $N^3\neq 0$ on $H^3(X)$. Then the special fibre $Y$ of $\mathscr{X}$ is a combinatorial Calabi--Yau of Type IV. 
\end{theorem}

As before, we will only treat the case $\mathrm{char}(F)=0$ and $\ell\neq p$, the others are handled identically. We may also assume that $k=\overline{k}$. Let $V_i$ denote the components of $Y$, $S_{ij}$ the double surfaces, $C_{ijk}$ the triple curves and $P_{ijkl}$ the quadruple points. Write $Y^{(0)}=\CMcoprod_i V_i$, $Y^{(1)}=\CMcoprod_{ij} S_{ij}$ et cetera. The only point where the hypothesis on the nilpotency index of $N$ is used is to prove the following lemma.

\begin{lemma} Suppose that $N^3\neq 0$. Then $Y$ has `maximal intersection', i.e. there exists a quadruple point $P_{ijkl}$.  
\end{lemma}

\begin{proof} If there is no quadruple point $P_{ijkl}$ then $Y^{(3)}=\emptyset$. Let $W_n$ denote the weight filtration on $H^3_\et(X_{\overline{F}},\Q_\ell)$, so that $W_{-1}=0$ and $W_6=H^3_\et(X_{\overline{F}},\Q_\ell)$. The monodromy operator $N^3$ sends $W_i$ into $W_{i-6}$, in particular $N^3(H^3_\et(X_{\overline{F}},\Q_\ell))\subset W_0$. But $Y^{(3)}=\emptyset$ implies that $W_0=0$ and hence $N^3=0$.
\end{proof}

Note that we do not need to know the weight-monodromy conjecture in order for the lemma to hold, we simply need to know compatibility of $N$ with the weight filtration.

For each $i$ we will let $D_i=\sum_{j\neq i} S_{ij}$, so that by the assumption $\omega_\mathcal{X}\cong \mathcal{O}_{\mathcal{X}}$ and the adjunction formula we have $-K_{V_i}=D_i$ for all $i$. Similarly setting $E_{ij}=\sum_{k\neq i,j} C_{ijk}$ we obtain $-K_{S_{ij}}=E_{ij}$ and setting $F_{ijk}=\sum_{l\neq ijk} P_{ijkl}$ we can see that $-K_{C_{ijk}}=F_{ijk}$. The lemma shows that there exists some $V_{i}$ containing a quadruple point, and the first key step in proving Theorem \ref{CY3} is showing that this is actually true for every $i$. The main ingredient in this is the following.

\begin{proposition} \label{divcon} Let $(V,D)$ be a pair consisting of a smooth projective threefold $V$ over $\overline{k}$ and a non-empty strict normal crossings divisor $D\subset V$. Assume that $K_V+D=0$, and that $D$ is disconnected. Then $D$ consists of two disjoint irreducible components $D_1$ and $D_2$.
\end{proposition}

\begin{remark} The corresponding result for surfaces follows from Lemma \ref{paac}.
\end{remark}

\begin{proof} The characteristic 0 version of this result is Proposition 4.37 of \cite{Kol13}. However, thanks to the proof of the Minimal Model Program for threefolds in characteristic $p>5$, in particular the connectedness principle and the existence of Mori fibre spaces in \cite{Bir13,BW14} the same proof works here. So we will run the MMP on the smooth 3-fold $V$. It follows from Theorem 1.7 of \cite{BW14} that this terminates in a Mori fibre space $p:V^*\rightarrow S$, and by the connectedness principle (Theorem 1.8 of \cite{Bir13}) it suffices to prove that the strict transform $D^*\subset V^*$ consists of 2 irreducible components. Now we simply follow the proof of Proposition 4.37 of \cite{Kol13}, which goes as follows.

We know that there exists some component $D_1^*\subset D^*$ which positively intersects the ray contracted by $p$. Choose another component $D_2^*\subset D^*$ disjoint from $D^*_1$, and choose some fibre $F_s$ of $p$ meeting $D_2^*$. Since $D_2^*$ is disjoint from $D_1^*$, it follows that it cannot contain $F_s$, and hence intersects $F_s$ positively. Hence both $D_1^*$ and $D_2^*$ are $p$-ample, intersecting the contracted ray positively. Hence the generic fibre of $p$ is of dimension $1$, and is a regular (not necessarily smooth) Fano curve. It then follows that if we choose a general fibre $F_g$ of $p$, then $D_i^*\cdot F_g=1$ for $i=1,2$ and all other components of $D^*$ are $p$-vertical, hence trivial as claimed.
\end{proof}

\begin{corollary} Every component of $Y$ contains a quadruple point.
\end{corollary}

\begin{proof} By connectedness of $Y$ it suffices to show that each neighbour of $V_{i}$ also contains a quadruple point. Note that by Proposition \ref{divcon} the divisor $D_{i}$ is connected, by hypothesis there exists a double surface $S_{ij}$ in $D_i$ containing a quadruple point, and hence it suffices to show that each double surface $S_{ik}$ meeting $S_{ij}$ contains a quadruple point. But if not, then $C_{ijk}$ would form a connected component of $E_{ij}$ and hence again applying Lemma \ref{paac} we would see that $S_{ij}$ could not contain a quadruple point. Therefore $S_{ik}$ must contain a quadruple point, and we are done.
\end{proof}

Of course this also shows that each double surface $S_{ij}$ contains a quadruple point, hence by repeatedly applying Lemma \ref{paac} we can conclude that each surface $S_{ij}$ and each curve $C_{ijk}$ is rational. We may therefore see as in the proof of Theorem \ref{crass} that the dual graph of each $D_i$ is a closed 2-manifold. Moreover, applying the MMP to each $V_i$ produces a Mori fibre space $W_i\rightarrow Z_i$, such that the divisor $D_i=\sum_{j\neq i}S_{ij}$ dominates $Z_i$. Therefore $V_i$ has the form described in Definition \ref{ccyt4}.

Finally, to show that the dual graph $\Gamma$ is a 3-sphere, we consider, for every vertex $\gamma$, corresponding to a component $V_i$ of $Y$, the `star' of $\gamma$, i.e. the subcomplex of $\Gamma$ consisting of those cells meeting $\gamma$. This is a cone over the dual graph of $D_i$, hence $\Gamma$ is a closed 3-manifold. 

\begin{proposition} The dual graph $\Gamma$ is simply connected.
\end{proposition}

\begin{proof} Let $\C_p$ denote the completion of the algebraic closure of $F$, and $\mathcal{O}_{\C_p}$ its ring of integers. Let $\frak{X}$ denote the base change to $\mathcal{O}_{\C_p}$ of the $\pi$-adic completion of $\mathscr{X}$, this $\frak{X}$ is polystable over $\mathcal{O}_{\C_p}$ in the sense of Definition 1.2 of \cite{Ber99}. Let $X_{\C_p}^\mathrm{an}$ denote the generic fibre of $\frak{X}$, considered as a Berkovich space, or in other words the analytification of the base change of $X$ to $\C_p$.

Let $\pi_1^\et(X^\mathrm{an}_{\C_p})$ denote the \'etale fundamental group of $X_{\C_p}^\mathrm{an}$ in the sense of \cite{dJ95a}, and by $\pi_1^\mathrm{top}(X^\mathrm{an}_{\C_p})$ the fundamental group of the underlying topological space of $X_{\C_p}^\mathrm{an}$. Theorem 2.10(iii) of \cite{dJ95a} together with rigid analytic GAGA shows that the profinite completion of $\pi_1^\et(X^\mathrm{an}_{\C_p})$ is trivial, since it is isomorphic to the algebraic \'etale fundamental group $\pi_1^\et(X_{\C_p})$ of $X_{\C_p}$, and $X$ is Calabi--Yau. Next, by Remark 2.11 of \cite{dJ95a} together with Theorem 9.1 of \cite{Ber99} we have a surjection $\pi_1^\et(X^\mathrm{an}_{\C_p})\rightarrow \pi_1^\mathrm{top}(X^\mathrm{an}_{\C_p})$ and hence the profinite completion of $\pi_1^\mathrm{top}(X^\mathrm{an}_{\C_p})$ is trivial. 

Now by Theorem 8.2 of \cite{Ber99} we have $\pi_1(\Gamma)\cong \pi_1^\mathrm{top}(X^\mathrm{an}_{\C_p})$ and hence the profinite completion of $\pi_1(\Gamma)$ is trivial. Since $\Gamma$ is a 3-manifold, we may finally apply \cite{Hem87} to conclude that $\pi_1(\Gamma)$ is trivial as claimed.
\end{proof}

We may now conclude the proof of Theorem \ref{CY3} using the Poincar\'e conjecture. In fact, if we know that the weight monodromy conjecture holds, then we have the following converse.

\begin{proposition} Let $\mathscr{X}$ be a strictly semistable $R$-scheme whose generic fibre is a Calabi--Yau threefold $X$, such that $\omega_\mathscr{X}\cong \mathcal{O}_{\mathscr{X}}$. Assume that the special fibre $Y$ is a combinatorial Calabi--Yau  of Type IV. If the weight monodromy conjecture holds for $H^3(X)$, then $N^3\neq 0$.
\end{proposition}

\begin{proof} Again, we assume that $\ell\neq p$, the other cases are handled similarly. Consider the weight spectral sequence $E_r^{p,q}$ for $\mathscr{X}$. The hypotheses imply that $N^3$ induces an isomorphism
\[ N^3:E_2^{-3,6}  \rightarrow E_2^{3,0} \]
and to show that $N^3\neq 0$ it therefore suffices to show that $E_2^{3,0}\neq 0$. Writing out the weight spectral sequence explicitly we see that we have an isomorphism
\[ E_2^{3,0}\cong H^3_\mathrm{sing}(\Gamma,\Q_\ell) \]
where $\Gamma\simeq S^3$ is the dual graph of $Y$, and hence the claim follows.
\end{proof}

This is in particular the case if $\mathrm{char}(F)=p$ (when $\ell\neq p$ this is \cite{Ito05}, when $\ell=p$ it is Chapter 5 of \cite{LP16}) or $\mathrm{char}(F)=0$, $\ell\neq p$ and $X$ is a complete intersection in some projective space (which follows from \cite{Sch12}).

\section*{Acknowledgements}

B. Chiarellotto was supported by the grant MIUR-PRIN 2010-11 "Arithmetic Algebraic Geometry and Number Theory". C. Lazda was supported by a Marie Curie fellowship of the Istituto Nazionale di Alta Matematica. Both authors would like to thank the anonymous referee for a careful reading of the paper, and for suggesting several important improvements.

\bibliographystyle{mysty}
\bibliography{/Users/Chris/Dropbox/LaTeX/lib.bib}

\end{document}